\documentclass{article}
\usepackage{amsfonts}
\usepackage{amsmath}
\usepackage{eurosym}

\newtheorem{theorem}{Theorem}

\newtheorem{definition}[theorem]{Definition}

\newtheorem{lemma}[theorem]{Lemma}

\newtheorem{remark}[theorem]{Remark}

\newenvironment{proof}[1][Proof]{\noindent \textbf{#1.} }{\  \rule{0.5em}{0.5em}}
\input{tcilatex}
\begin{document}

\begin{center}
{\Large Note on Approximation of Truncated Baskakov Operators By Fuzzy
Numbers}{\LARGE \ }

\begin{equation*}
\end{equation*}

Ecem Acar, Sevilay K\i rc\i \ Serenbay, Saleem Yaseen Majeed Al Khalidy%
\begin{equation*}
\end{equation*}

Harran University, Department of Mathematics, Turkey

karakusecem@harran.edu.tr,sevilaykirci@gmail.com,saleem.yaseen@garmian.edu.krd\}
\end{center}


Subjclass[2010]:41A30, 41A46, 41A25.

\begin{center}
Abstract
\end{center}

In this paper, we firstly introduce nonlinear truncated Baskakov operators
on compact intervals and obtain some direct theorems. Also, we give the
approximation of fuzzy numbers by truncated nonlinear Baskakov operators. 

\bigskip 

\bigskip 

\bigskip 

\section{Introduction}

The concepts of fuzzy numbers and fuzzy arithmetic were introduced by Zadeh 
\cite{z}. The representation of fuzzy numbers by appropriate intervals that
depend mainly on the shape of their membership functions is an interesting
and important problem and has many applications in various fields. And it is
known that dealing with fuzzy numbers is often difficult because of the very
complex representation of the shapes of their membership functions. That is
to say, the interpretation and expression of fuzzy numbers are more
intuitive and more natural whenever the shapes of their membership functions
are simpler. Many studies have recently been published that investigate the
approximation of fuzzy numbers by trapezoidal or triangular fuzzy members
(see \cite{br5}-\cite{br26}).

The core topic of Korovkin type approximation theory is the approximation of
a continuous function by a series of linear positive operators (see \cite{k1}%
,\cite{k2}). Bede et al. \cite{bede1} have recently proposed nonlinear
positive operators in place of linear positive operators. Although the
Korovkin theorem fails for these nonlinear operators, they observed that
they behave similarly to linear operators in terms of approximation.

The main purpose of this paper is to use called Truncated Baskakov operator
of max-product kind which is given in the book \cite{bede1}, for
approximating fuzzy numbers with continuous membership functions. We will
show that these operators additionally maintain the quasi-concavity in a
manner analogous to the specific case of the unit interval. These results
turn out to be particularly useful in the approximation of fuzzy numbers
since they will enable us to construct fuzzy numbers with the same support
in a straightforward manner. Additionally, these operators provide a good
order of approximation for the (non-degenerate) segment core.

Baskakov \cite{b1} introduced the positive and linear operators, which are
typically associated to functions that are bounded and uniformly continuous
to $f\in C\left[ 0,+\infty \right] $ and specified by 
\begin{equation}
V_n\left(f \right)(x)=(1+x)^{-n}\sum_{k=0}^\infty \binom{n+k-1}{k}%
x^k(1+x)^{-k}f\left( \frac{k}{n}\right) , \  \forall n\in \mathbb{N}.
\end{equation}
It is known that the following pointwise approximation result (see \cite{f})
exists as: 
\begin{equation*}
\mid V_n\left(f \right)(x)-f(x)\mid \leq C\omega_2^{\varphi}\left( f;\sqrt{%
x(1+x)/n}\right), \ x\in \left[ 0,\infty \right) , n\in \mathbb{N}, 
\end{equation*}
where $\varphi(x)=\sqrt{x(1+x)}$ and $I=[0,\infty).$ In this case, \newline
$I_h=[h^2/(1-h)^2,+\infty), \ h\leq \delta<1.$ Additionally, $V_n\left(f
\right)$ preserves the monotonicity and convexity (of any order) of the
function $f$ on $[0, +\infty)$ (see \cite{lupas}).

The truncated Baskakov operators are defined 
\begin{equation*}
U_n(f)(x)=(1+x)^{-n}\sum_{k=0}^n\binom{n+k-1}{k}x^k(1+x)^{-k}f\left( \frac{k%
}{n}\right),\ f\in C[0,1]. 
\end{equation*}

Truncated Baskakov operator of max product kind $f:[0,1]\rightarrow \mathbb{R%
}$ are defined by (see \cite{bede1}) 
\begin{equation*}
U_n^{(M)}(f)(x)=\frac{\bigvee_{k=0}^nb_{n,k}(x)f\left( \frac{k}{n}\right)}{%
\bigvee_{k=0}^nb_{n,k}(x)}, \ x\in [0,1], n\in \mathbb{N}, \ n\geq 1, 
\end{equation*}
where $b_{n,k}(x)=\binom{n+k-1}{k}x^k(1+x)^{-n-k}$, $n\geq 1,$ $x\in [0,1].$
As it was proved in \cite{bede1}, Lemma 4.2.1, for any arbitrary function $%
f:[0,1]\rightarrow \mathbb{R}_+$ $U_n^{(M)}(f)(x)$ is positive, continuous
on $[0,1]$ and satisfies $U_n^{(M)}\left( f\right)(0)=f(0)$ for all $n\in 
\mathbb{N}, n\geq 2.$ In the paper \cite{bede2}, it was showed that the
order of uniform approximation in the whole class $C_+([0,1])$ of positive
continuous functions on $[0,1]$ cannot be improved, in the sense that there
exists a function $f \in C_+([0,1])$, for which the approximation order by
the truncated max-product Baskakov operator is $C\omega_1 (f,1/ \sqrt{n})$.
The fundamentally better order of approximation $\omega_1 (f,1/ n)$ was
attained for some functional subclasses, such as the nondecreasing concave
functions. Finally, some shape preserving properties were proved. In this
study, firstly we extend to an arbitrary compact interval the definition of
the truncated Baskakov operators of max-product kind, by proving that their
order of uniform approximation is the same as in the particular case of the
unit interval. Then, similarly to the particular case of the unit interval,
we are proved that these operators preserve the quasi-concavity. Since these
properties help us to generate in a simple way fuzzy numbers of the same
support, it turns out that these results are very suitable in the
approximation of fuzzy numbers.

\section{Preliminaries}

\begin{definition}
$\left[ 13,32\right] $\textbf{( fuzzy numbers) }A fuzzy subset $u$ of the
real line $%
\mathbb{R}
$ with membership function $\mu _{u}\left( x\right) :%
\mathbb{R}
\longrightarrow \left[ 0,1\right] $ is called a fuzzy number if:\medskip

\begin{enumerate}
\item $u$ is normal, i.e. $\exists x_{0}\in 
\mathbb{R}
$\ such that $\mu _{u}\left( x_{0}\right) =1;\medskip $

\item $u$ is fuzzy convex, i.e. $\mu _{u}\left( \lambda x+\left( 1-\lambda
\right) y\right) \geq \min \{ \mu _{u}\left( x\right) ,\mu _{u}\left(
y\right) \};\medskip $

\item $\mu _{u}$ is upper semicontinuous;and$\medskip $

\item $supp\left( u\right) $ is bounded, where supp$\left( u\right) =cl$ $%
\{x\in 
\mathbb{R}
\mid \mu _{u}\left( x\right) >0\},$where $\left( cl\right) $is the closure
operator .i.e. $supp\left( u\right) $ is compact.\bigskip
\end{enumerate}
\end{definition}

It is clear, for any fuzzy number $u$ there exist four numbers $%
t_{1},t_{2},t_{3},t_{4}\in 
\mathbb{R}
$ and two functions $l_{u},r_{u}:%
\mathbb{R}
\rightarrow \left[ 0,1\right] $ such that we can describe a membership

function $\mu _{u}$ in a following manner:\bigskip

$\mu _{u}(x)=\left \{ 
\begin{array}{c}
0 \\ 
l_{u}\left( x\right) \\ 
1 \\ 
r_{u}\left( x\right) \\ 
0%
\end{array}%
\right. 
\begin{array}{c}
\text{ }if\text{ \  \  \ }x<t_{1}\text{ \  \  \  \  \  \  \ } \\ 
if\text{ \  \  \ }t_{1}\leq x\leq t_{2} \\ 
if\text{ \  \  \ }t_{2}\leq x\leq t_{3} \\ 
if\text{ \  \  \ }t_{3}\leq x\leq t_{4} \\ 
\text{ \ }if\text{ \  \  \ }t_{4}<x\text{ \  \  \  \  \  \  \  \ }%
\end{array}%
$\bigskip

where $l_{u}:\left[ t_{1},t_{2}\right] \rightarrow \left[ 0,1\right] $ is
nondecreasing called the left side of a fuzzy number $u$ and $r_{u}:\left[
t_{3},t_{4}\right] \rightarrow \left[ 0,1\right] $ is nonincreasing called
the right side of a fuzzy number $u$.

Another representation of a fuzzy number is the so called $\alpha-cut$
representation also known as $LU$ parametric representation. In this case
the fuzzy number $u$ is given by a pair of functions $(u^-, u^+)$ where $%
u^-, u^+ : [0, 1]\rightarrow \mathbb{R}$ satisfy the following requirements:

\begin{itemize}
\item[i.] $u^-$ is nondecreasing;

\item[ii.] $u^+$ is nonincreasing;

\item[iii.] $u^-(1)\leq u^+(1)$
\end{itemize}

It is known that for $u = (u^-,u^+)$, we have $core(u) = [u^-(1),u^+(1)] $
and $supp(u) = [u^-(0),u^+(0)]$. An important connection between the
membership function of a fuzzy number $u$ and its parametric representation
is given by the following well known relations: 
\begin{equation*}
u^-(\alpha)=\inf \left \lbrace x\in \mathbb{R} : u(x)\geq \alpha \right
\rbrace 
\end{equation*}
\begin{equation*}
u^+(\alpha)=\sup \left \lbrace x\in \mathbb{R} : u(x)\geq \alpha \right
\rbrace , \alpha \in (0,1] 
\end{equation*}
and 
\begin{equation*}
[u^-(0),u^+(0)]=cl( \left \lbrace x\in \mathbb{R} : u(x)> 0 \right \rbrace ) 
\end{equation*}
where $cl$ denotes the closure operator. Moreover,in the case when $u$ is
continuous with $supp(u) = [a,b]$ and $core(u) =[c,d]$, then it can be
proved that 
\begin{equation*}
u(u^-(\alpha))=\alpha, \forall \alpha \in [a,c] 
\end{equation*}
\begin{equation*}
u(u^+(\alpha))=\alpha, \forall \alpha \in [d,b] 
\end{equation*}

Therefore a fuzzy number $u\in \mathbb{R}_{F}$ is completely determined by
the end points of the intervals%
\begin{equation*}
\left[ u\right] _{\alpha }=\left[ u^{-}\left( \alpha \right) ,u^{+}\left(
\alpha \right) \right] ,\forall \alpha \in \lbrack 0,1]. 
\end{equation*}%
Therefore we can identify a fuzzy number $u\in \mathbb{R}_{F}$ with its
parametric representation 
\begin{equation*}
\left \{ \left( u^{-}\left( \alpha \right) ,u^{+}\left( \alpha \right)
\right) \mid 0\leq \alpha \leq 1\right \} 
\end{equation*}
and we can write $u=\left( u^{-},u^{+}\right) .$

Then, we correlateto fuzzy numbers some important chracterisctics.

\textit{The expected interval} $EI(u)$ of a fuzzy number $u$ where $\left[ u%
\right] _{\alpha }=\left[ u^{-}\left( \alpha \right) ,u^{+}\left( \alpha
\right) \right] $ $\forall \alpha \in \lbrack 0,1]$, is defined by (see $%
[14,30]$)\bigskip

$EI(u)=\left[ \underset{0}{\overset{1}{\int }}u^{-}\left( \alpha \right)
d\alpha ,\underset{0}{\overset{1}{\int }}u^{+}\left( \alpha \right) d\alpha %
\right] $\bigskip

and \textit{the expected value }of the fuzzy number $u$ by (see $[30]$%
)\bigskip

$EV\left( u\right) =\frac{1}{2}\underset{0}{\overset{1}{\int }}\left(
u^{-}\left( \alpha \right) +u^{+}\left( \alpha \right) \right) d\alpha $%
\bigskip

\textit{In particular}, $EI(u)$ can be considered as a fuzzy number and more
precisely an interval fuzzy number.\bigskip

\textit{The width} of a fuzzy number $u$ is given by (see $[28]$)\bigskip

$wid\left( u\right) =\underset{0}{\overset{1}{\int }}\left( u^{+}\left(
\alpha \right) -u^{-}\left( \alpha \right) \right) d\alpha $\bigskip

Let $s:\left[ 0,1\right] \rightarrow \left[ 0,1\right] $ is a nondecreasing
function such that $s(0)=0$ and $s(1)=1$, sometimes called a reduction
function.

\textit{The value} of $u$ with respect to $s$ is given by (see $[19]$%
)\bigskip

$Val_{s}\left( u\right) =\underset{0}{\overset{1}{\int }}s\left( \alpha
\right) \left( u^{-}\left( \alpha \right) +u^{+}\left( \alpha \right)
\right) d\alpha $\bigskip

and \textit{the ambiguity }of a fuzzy number $u$ with respectto $s$ is
defined by (see $[19]$)\bigskip

$Amb_{s}\left( u\right) =\underset{0}{\overset{1}{\int }}s\left( \alpha
\right) \left( u^{+}\left( \alpha \right) -u^{-}\left( \alpha \right)
\right) d\alpha $\bigskip

Also we have the particular case of reduction function $s$ for $s\left(
\alpha \right) =\alpha ^{r}$ , $r\in 
\mathbb{N}
$ and $\alpha \in \left[ 0,1\right] $, then we denote $Val_{sr}\left(
u\right) =Val_{r}\left( u\right) $ and $Amb_{sr}\left( u\right)
=Amb_{r}\left( u\right) $, (see $[19]$) i.e.\bigskip

$Val_{r}\left( u\right) =\underset{0}{\overset{1}{\int }}\alpha ^{r}\left(
u^{-}\left( \alpha \right) +u^{+}\left( \alpha \right) \right) d\alpha $%
\bigskip

and\bigskip

$Amb_{r}\left( u\right) =\underset{0}{\overset{1}{\int }}\alpha ^{r}\left(
u^{+}\left( \alpha \right) -u^{-}\left( \alpha \right) \right) d\alpha $%
\bigskip

\section{Truncated Baskakov Operators defined on compact intervals}

From Theorem \ref{t6}, one can see that the order of uniform approximation
of the fuzzy number $u$ by $\widetilde{F}_{n}^{\left( M\right) }\left( u;%
\left[ \vartheta, b\right] \right)$ on $\mathbb{R}$ is $\omega_1(u;1/\sqrt{n}%
)_{[\vartheta, b]}.$

All throughout this paper, we indicate the continuous function space defined
on interval $I$ by $C(I)$ and the positive continuous function space defined
on interval $I$ by $C_+(I)$. From the result of Weierstrass theorem (see 
\cite{w}), $P(x)$ converges to continuous function $f(x)$ in the interval $%
[0,1]$, we just have to move functions from $[0,1]$ to an arbitrary interval 
$[a,b].$ In fact, let consider the continuous function $g:[0,1]\rightarrow 
\mathbb{R}$ and the function $f(x)$ is continuouson $[a,b],$ we put $%
g(y)=f(a+(b-a)y).$

If u is a continuous fuzzy number with $supp(u)=[a,b]$, $a<b$ and $%
core(u)=[c,d]$, $c<d.$ then we can define

\begin{equation*}
\widetilde{U}_{n}\left( u\right)(x) =\left \{ 
\begin{array}{l}
0,\text{ }x\notin \left[a,b\right] \\ 
U_{n}\left( u;\left[a,b\right] \right)=\sum_{k=0}^n b_{n,k}(x)u\left( a+(b-a)%
\frac{k}{n}\right) ,x \in \left[a,b\right] .%
\end{array}%
\right. 
\end{equation*}
where $b_{n,k}(x)=\binom{n+k-1}{k}\left( \frac{x-a}{b-a}\right) ^k(\frac{%
b-2a+x}{b-a})^{-n-k}$.

\begin{equation*}
\label{f7} U_n^{(M)}\left(f;[a,b] \right) (x)= \frac{%
\bigvee_{k=0}^nb_{n,k}(x)f\left( a+(b-a)\frac{k}{n}\right)}{%
\bigvee_{k=0}^nb_{n,k}(x)}, \ x\in [a,b] 
\end{equation*}
where $b_{n,k}(x)=\binom{n+k-1}{k}\left( \frac{x-a}{b-a}\right) ^k(\frac{%
b-2a+x}{b-a})^{-n-k}$.

We have $\sum_{k=0}^nb_{n,k}(x)=1$ for all $x\in [a,b]$ so we get that $%
\bigvee_{k=0}^nb_{n,k}(x)> 0$ which means that $U_n^{(M)}\left(f;[a,b]
\right)$ is well defined.

Since the maximum of a finite number of continuous functions is a continuous
function, we get that for any $f\in C_+\left( [a,b]\right),$ $%
U_n^{(M)}\left(f;[a,b] \right)\in C_+\left( [a,b]\right).$ In this section
we will prove that $U_n^{(M)}: C_+\left( [a,b]\right)\rightarrow C_+\left(
[a,b]\right)$ has the same order of uniform approximation as the linear
Baskakov operator and that it preserves the quasi-concavity too. Firstly, we
need the following results and definitions.

\begin{theorem}
\label{t2}

\begin{itemize}
\item[i.] (\cite{bede2},Theorem 4.1)If $f:[0,1]\rightarrow \mathbb{R}_+$ is
continuous then we have the estimate 
\begin{equation*}
\mid U_n^{(M)}\left(f;[0,1] \right) (x)-f(x)\mid \leq 24\omega_1\left(f;%
\frac{1}{\sqrt{n+1}}\right)_{[0,1]}, \ n\in \mathbb{N}, \ n\geq 2, \ x\in
[0,1].
\end{equation*}

\item[ii.] (\cite{bede2}, Corollary 4.5) If $f:[0,1]\rightarrow \mathbb{R}_+$
is concave then we have the estimate 
\begin{equation*}
\mid U_n^{(M)}\left(f;[0,1] \right) (x)-f(x)\mid \leq 2\omega_1\left(f;\frac{%
1}{n}\right)_{[0,1]}, \ n\in \mathbb{N}, \ x\in [0,1].
\end{equation*}
\end{itemize}
\end{theorem}

\begin{theorem}
\label{t3}  Let us consider the function $f:[0,1]\rightarrow \mathbb{R}_+$
and let us fix $n \in \mathbb{N}, \ n\geq 2.$ Suppose, in addition, that
there exists $c\in [0,1]$ such that $f$ is nondecreasing on $[0,c]$ and
nonincreasing on $[c,1].$ Then, there exists $c^{\prime }\in [0,1]$ such
that $U_n^{(M)}\left(f \right)$ is nondecreasing on $[0,c^{\prime }]$ and
nonincreasing on $[c^{\prime },1].$ In addition, we have $\mid c-c^{\prime
}\mid \leq 1/(n+1)$ and $\mid U_n^{(M)}\left(f \right)(c)-f(c)\mid \leq
\omega_1\left( f;1/(n+1)\right).$
\end{theorem}

\begin{proof}
Let $j_c\in \left \lbrace 0,1,\cdots, n-2 \right \rbrace $ be such that $%
\left[\frac{j_c}{n-1},\frac{j_c+1}{n-1} \right].$ We will study the
monotonicity on each interval of the form $\left[\frac{j}{n-1},\frac{j+1}{n-1%
} \right],$ $j\in \left \lbrace 0,1,\cdots, n-2 \right \rbrace.$ Then by the
continuity of $U_n^{(M)}\left(f\right),$ we will be able to determine the
monotoncity of $U_n^{(M)}\left(f\right)$ on $[0,1].$ Let us choose arbitrary 
$j\in \left \lbrace 0,1,\cdots, j_c-1 \right \rbrace$ and $x\in \left[\frac{j%
}{n-1},\frac{j+1}{n-1} \right].$ By the monotonicity of $f$, it follows that 
$f\left(\frac{j}{n} \right)\geq f\left(\frac{j-1}{n} \right)\geq \cdots \geq
f(0).$ From \cite{bede2}(proof of Lemma 3.2) the following assertions hold: 
\begin{equation}  \label{yıldız}
\text{If} \ j\leq k\leq k+1\leq n \  \text{then} 1\geq m_{k,n,j}(x)\geq
m_{k+1,n,j}(x),
\end{equation}
\begin{equation}
\text{If} \ 0\leq k\leq k+1\leq j \  \text{then} m_{k,n,j}(x)\leq
m_{k+1,n,j}(x)\leq 1.
\end{equation}
Therefore, it is easily follows that $f_{j,n,j}(x)\geq f_{j-1,n,j}(x)\geq
\cdots \geq f_{0,n,j}(x).$ From \cite{bede2} lemma 3.4, it follows that 
\begin{equation*}
U_n^{(M)}\left(f\right)=\bigvee_{k=j}^n f_{k,n,j}(x).
\end{equation*}
Since $U_n^{(M)}\left(f\right)$ is defined as the maximum of nondecreasing
functions, it follows that it is nondecreasing on $\left[\frac{j}{n-1},\frac{%
j+1}{n-1} \right].$ Taking into account the continuity of $%
U_n^{(M)}\left(f\right),$ it is immediate that $f$ is nondecreasing on $%
\left[0,\frac{j_c}{n-1} \right].$ Now let us choose arbitrary $j\in \left
\lbrace j_c+1,\cdots, n-2 \right \rbrace$ and $x\in \left[\frac{j}{n-1},%
\frac{j+1}{n-1} \right].$ By the monotonicity of $f,$ it follows that $%
f\left(\frac{i}{n} \right)\geq f\left(\frac{i+1}{n} \right)\geq \cdots \geq
f(1).$ It easily follows that $U_n^{(M)}\left(f\right)(x)=\bigvee_{k=0}^j
f_{k,n,j}(x)$ from the assertion (\ref{yıldız}). Since $%
U_n^{(M)}\left(f\right)$ is defined as the maximum of nonincreasing
functions, it follows that it is nonincreasing on $\left[\frac{j}{n-1},\frac{%
j+1}{n-1} \right].$ Taking into account the continuity of $%
U_n^{(M)}\left(f\right)$ , it is immediate that $f$ is nonincreasing on $%
\left[\frac{j_c+1}{n-1},1 \right].$ Finially let us discuss the case when $%
j=j_c.$ If $\frac{j}{n-1}\leq c,$ then by the monotoncity of $f$ it follows
that $f\left(\frac{j_c}{n} \right)\geq f\left(\frac{j_c-1}{n} \right)\geq
\cdots \geq f(0).$ Therefore , in the case we obtain $f$ is nondecreasing on 
$\left[\frac{j_c}{n-1},\frac{j_c+1}{n-1} \right].$ It follows that $f$ is
nondecreasing on $\left[0,\frac{j_c+1}{n-1} \right]$ and nonincreasing on $%
\left[\frac{j_c+1}{n-1},1 \right].$ In addition, $c^{\prime }=\frac{j_c+1}{%
n-1}$ is the maximum point of $U_n^{(M)}\left(f\right)$ and it is easy to
check that $\mid c-c^{\prime }\mid \leq \frac{1}{n-1}.$ If $j_c/n\geq c$
then by the monotonicity of $f$ it follows that$f\left(\frac{j_c}{n}
\right)\geq f\left(\frac{j_c+1}{n} \right)\geq \cdots \geq f(1).$ Therefore,
in this case we obtain that $f$ is nonincreasing on $\left[\frac{j_c}{n-1},%
\frac{j_c+1}{n-1} \right].$ It follows that $f$ is nondecreasing on $[0, 
\frac{j_c}{n-1}]$  and nonincreasing on $[\frac{j_c}{n-1},1].$ In addition,$%
c^{\prime }=\frac{j_c}{n-1}$ is the maximum point of $U_n^{(M)}\left(f\right)
$ and it is easy to check that $\mid c-c^{\prime }\mid \leq \frac{1}{n-1}.$

We prove now the last part of the theorem. First, let us notice that $%
U_n^{(M)}\left(f\right)\leq f(c)$ for all $x\in [0,1].$ Indeed, this is
immediate by the definition of $U_n^{(M)}\left(f\right)$ and by the fact
that $c$ is the global maximum point of $f$. This implies  
\begin{equation*}
\begin{split}
\mid U_n^{(M)}\left(f\right)(c)-f(c)\mid =& f(c)-
U_n^{(M)}\left(f\right)(c)=f(c)-\bigvee_{k=0}^n f_{k,n,j_c}(c) \\
& \leq f_{j_c,n,j_c}(c)=f(c)-f\left( \frac{j_c}{n}\right).
\end{split}%
\end{equation*}
Since $c,\frac{j_c}{n} \in \left[\frac{j}{n-1},\frac{j+1}{n-1} \right],$ we
easily get $f(c)-f\left( \frac{j_c}{n}\right)\leq \omega_1\left(f;?1/n+1?
\right) $ and the theorem is proved completely.
\end{proof}

\begin{definition}
(\cite{gal}) Let $f:[a,b]\rightarrow \mathbb{R}$ be continuous on $[a,b]$.
The function $f$ is called:

\begin{itemize}
\item[i.] quasi-convex if $f(\lambda x+(1-\lambda)y)\leq \max \left \lbrace
f(x),f(y)\right \rbrace, $ for all $x,y\in [a,b],$ $\lambda \in [0,1]$

\item[ii.] quasi-concave, if $-f$ is quasi-convex.
\end{itemize}
\end{definition}

\begin{remark}
\label{r3.1} By \cite{pop}, the continuous function $f$ is quasi-convex on $%
[a,b]$ equivalently means that there exists a point $c\in [a,b]$ such that $f
$ is nonincreasing on $[a,c]$ and nondecreasing on $[c,b].$ From the above
definiton, we easily get that the function $f$ is quasi-concave on $[a,b],$
equivalently means that there exists a point $c\in [a,b]$ such that $f$ is
nondecreasing on $[a,c]$ and nonincreasing on $[c,b].$
\end{remark}

We can now present the main results of this section.

\begin{theorem}
\label{t6}

\begin{itemize}
\item[i.] If $a,b\in \mathbb{R}$, $a<b$ and $f:[a,b]\rightarrow \mathbb{R}_+$
is continuous then we have the estimate 
\begin{equation*}
\mid U_n^{(M)}\left(f;[a,b] \right) (x)-f(x)\mid \leq
24\left([b-a]+1\right)\omega_1\left(f;\frac{1}{\sqrt{n+1}}\right)_{[a,b]}, \
n\in \mathbb{N}, \ n\geq 2, \ x\in [a,b].
\end{equation*}

\item[ii.] If $f:[a,b]\rightarrow \mathbb{R}_+$ is concave then we have the
estimate 
\begin{equation*}
\mid U_n^{(M)}\left(f;[a,b] \right) (x)-f(x)\mid \leq
2\left([b-a]+1\right)\omega_1\left(f;\frac{1}{n}\right)_{[a,b]}, \ n\in 
\mathbb{N}, \ x\in [a,b].
\end{equation*}
\end{itemize}
\end{theorem}

\begin{proof}
Let take the function $h(y)$ is continuous on $[0,1]$ as $h(y)=f(a+(b-a)y).$
It is easy to check that $h(\frac{k}{n})=f\left( a+k.\frac{(b-a)}{n}\right) $
for all $k\in \left \lbrace 0,1,\cdots,n\right \rbrace.$ Now let choose
arbitrary $x\in [a,b]$ and let $y\in [0,1]$ such that $x=a+(b-a)y.$ This
implies $y=(x-a)/(b-a)$ and $1+y=\frac{b+x-2a}{b-a}.$ From these equalities
and noting the expressions for $h\left( \frac{k}{n}\right),$ we obtain $%
U_n^{(M)}\left(f;[a,b] \right)(x)=U_n^{(M)}\left(h;[0,1] \right)(y).$ From
Theorem \ref{t2} 
\begin{equation}
\mid U_n^{(M)}\left(f;[a,b] \right)(x)-f(x)\mid=\mid U_n^{(M)}\left(h;[0,1]
\right)(y)-h(y)\mid \leq 24\omega_1\left(h;\frac{1}{\sqrt{n+1}}%
\right)_{[0,1]}.
\end{equation}
Since $\omega_1\left(h;\frac{1}{\sqrt{n+1}}\right)_{[0,1]}\leq
\omega_1\left(f;\frac{b-a}{\sqrt{n+1}}\right)_{[a,b]}$ and the property $%
\omega_1\left( f;\lambda \delta \right)_{[a,b]}\leq \left(
[\lambda]+1\right)\omega_1\left( f;\delta \right)_{[a,b]},$ we obtain $%
\omega_1\left(h;\frac{1}{\sqrt{n+1}}\right)_{[0,1]}\leq \left([b-a]+1\right)
\omega_1\left(f;\frac{1}{\sqrt{n+1}}\right)_{[a,b]}$ which proves (i).

Keeping the notation from the above point (i), we get $U_n^{(M)}%
\left(f;[a,b] \right)(x)=U_n^{(M)}\left(h;[0,1] \right)(y),$ where $%
h(y)=f(a+(b-a)y)=f(x)$ for all $y\in [0,1].$ The last equality is equivalent
to $f(u)=h\left( \frac{u-a}{b-a}\right) $ for all $u\in [a,b].$ Writing now
the property of concavity for $f,$ $f(\lambda u_1+(1-\lambda)u_2)\geq
\lambda f(u_1)+(1-\lambda)f(u_2),$ for all $\lambda \in [0,1],$ $u_1,u_2 \in
[a,b],$ in terms of $h$ can be written as 
\begin{equation*}
h\left( \lambda \frac{u_1-a}{b-a}+(1-\lambda)\frac{u_2-a}{b-a}\right) \geq
\lambda h \left(\frac{u_1-a}{b-a}\right)+(1-\lambda)h\left(\frac{u_2-a}{b-a}%
\right).
\end{equation*}
Denoting $y_1=\frac{u_1-a}{b-a}\in [0,1]$ and $y_2=\frac{u_2-a}{b-a}\in [0,1]
$ this immediately implies the concavity of h on $[0,1]$ . Then, by Theorem %
\ref{t2} (ii), we get 
\begin{equation*}
\mid U_n^{(M)}\left(f;[a,b] \right)(x)-f(x)\mid=\mid U_n^{(M)}\left(h;[0,1]
\right)(y)-h(y)\mid \leq 2\omega_1\left(h;\frac{1}{n}\right)_{[0,1]}.
\end{equation*}
Reasoning now exactly as in the above point (i), we get the desired
conclusion.
\end{proof}

\begin{theorem}
\label{t7} Let us consider the function $f:[a,b]\rightarrow \mathbb{R}_+$
and let us fix $n\in \mathbb{N}$, $n\geq 1.$ Suppose in addition that there
exists $c\in [a,b]$ such that $f$ is nondecreasing on $[a,c]$ and
nonincreasing on $[c,b].$ Then, there exists $c^{\prime }\in [a,b]$ such
that $U_n^{(M)}\left(f;[a,b]\right)$ is nondecreasing on $[a,c^{\prime }]$
and nonincreasing on $[c^{\prime },b].$ In addition we have $\mid
c-c^{\prime }\mid \leq (b-a)/(n+1)$ and $\mid
U_n^{(M)}\left(f;[a,b]\right)(c)-f(c)\mid \leq
\left([b-a]+1\right)\omega_1\left(f;\frac{1}{\sqrt{n+1}}\right)_{[a,b]}.$
\end{theorem}

\begin{proof}
We construct the function $h$ as in the previous theorem. Let $c_1\in [0,1]$
be such that $h(c_1)=c$ Since $h$ is the composition between $f$ and the
linear nondecreasing function $t\rightarrow a+(b-a)t,$ we get that $h$ is
nondecreasing on $[0,c_1]$ and nonincreasing on $[c_1,1].$ By Theorem \ref%
{t3} it results that there exists $c^{\prime }\in [0,1]$ such that $%
U_n^{(M)}\left(h;[0,1]\right)$ is nondecreasing on $[0,c_1^{\prime }]$ and
nonincreasing on $[c_1^{\prime },1]$ and in addition we have $\mid
U_n^{(M)}\left(h;[0,1]\right)(c_1)-h(c_1)\mid \leq \omega_1\left(h;1/n+1??
\right) $ and $\mid c_1-c_1^{\prime }\mid \leq 1/(n+1).$ Let $c^{\prime
}=a+(b-a)c_1^{\prime }.$ If $x_1,x_2 \in [a,c^{\prime }]$ with $x_1 \leq x_2$
then let $y_1,y_2 \in [0,c_1^{\prime }]$ be such that $x_1=a+(b-a)y_1$ and $%
x_2=a+(b-a)y_2.$ Than it follows that $U_n^{(M)}\left(f;[a,b]%
\right)(x_1)=U_n^{(M)}\left(h;[0,1]\right)(y_1)$ and $U_n^{(M)}\left(f;[a,b]%
\right)(x_2)=U_n^{(M)}\left(h;[0,1]\right)(y_2).$ the monotonicity of $%
U_n^{(M)}\left(h;[0,1]\right)$ implies $U_n^{(M)}\left(h;[0,1]\right)(y_1)%
\leq U_n^{(M)}\left(h;[0,1]\right)(y_2)$ that is $U_n^{(M)}\left(f;[a,b]%
\right)(x_1)\leq U_n^{(M)}\left(f;[a,b]\right)(x_2).$ We thus obtain that $%
U_n^{(M)}\left(f;[a,b]\right)$ is nondecreasing on $[a,c^{\prime }].$ Using
the same type of reasoning, we obtain that $U_n^{(M)}\left(f;[a,b]\right)$
is nonincreasing on $[c^{\prime }, b].$ For the rest of the proof, noting
that $\mid c_1-c_1\hat{a}\euro 
{{}^2}
\mid \leq 1/(n + 1)$ we get $\mid c-c^{\prime }\mid= \mid
(b-a)(c_1-c_1^{\prime })\mid \leq (b-a)/(n + 1).$ Finally, noting that 
\begin{equation*}
\mid U_n^{(M)}\left(h;[0,1]\right)(c_1)-h(c_1)\mid \leq \omega_1\left(h;%
\frac{1}{n+1}\right)_{[0,1]}
\end{equation*}
and taking into account that $\omega_1\left(h;\frac{1}{n+1}%
\right)_{[0,1]}\leq \left([b-a]+1\right)\omega_1\left(f;\frac{1}{n+1}%
\right)_{[a,b]},$ we obtain 
\begin{equation*}
\begin{split}
\mid U_n^{(M)}\left(f;[a,b]\right)(c)-f(c)\mid =& \mid
U_n^{(M)}\left(h;[0,1]\right)(c_1)-h(c_1)\mid \\
\leq &\omega_1\left(h;\frac{1}{\sqrt{n+1}}\right)_{[0,1]} \leq
\left([b-a]+1\right)\omega_1\left(f;\frac{1}{\sqrt{n+1}}\right)_{[a,b]} \\
\end{split}%
\end{equation*}
and the proof is complete.
\end{proof}

\begin{remark}
From the above theorem and Remark \ref{r3.1}, it results that if $%
f:[a,b]\rightarrow \mathbb{R}_+$ is continuous and quasi-concave then $%
U_n^{(M)}(f)$ is quasi-concave too. As we have mentioned in the
Introduction, for functions in the space $C_+([0,1])$, $U_n^{(M)}$ preserves
the monotonicity and the quasi-convexity. Reasoning similar as in the proof
of Theorem \ref{t7}, it can be proved that these preservation properties
hold in the general case of the space $C_+([a, b])$.
\end{remark}

\section{Applications to the approximation of fuzzy numbers}

\begin{lemma}
\label{12} Let $a,b\in \mathbb{R}$, $a< b.$ For $n\in \mathbb{N}, \ n\geq 2$
and $k\in \left \lbrace 0,1,\cdots, n\right \rbrace , j\in \left \lbrace
0,1,\cdots, n-2\right \rbrace$ and $x\in \left(a+j.(b-a)/(n-1),
a+(j+1).(b-a)/(n-1) \right).$ Let $m_{k,n,j}(x)=\frac{b_{n,k}(x)}{b_{n,j}(x)}%
,$ where recall that $b_{n,k}(x)=\binom{n+k-1}{k}\left( \frac{x-a}{b-a}%
\right) ^k(\frac{b-2a+x}{b-a})^{-n-k}.$ Then $m_{k,n,j}(x)\leq 1.$
\end{lemma}

\begin{proof}
Without any loss of generality we may suppose that $a = 0$ and $b = 1$,
because using the same reasoning as in the proof of Theorems \ref{t6} and %
\ref{t7} we easily obtain the conclusion of the lemma in the general case.
So, let us fix $x\in (j/(n-1),(j+1)/(n-1)).$ According to Lemma 3.2 in \cite%
{bede2} we have 
\begin{equation*}
m_{0,n,j}(x)\leq m_{1,n,j}(x)\leq \cdots \leq m_{j,n,j}(x),
\end{equation*}
\begin{equation*}
m_{j,n,j}(x)\geq m_{j+1,n,j}(x)\geq \cdots \geq m_{n,n,j}(x).
\end{equation*}

Since $m_{j,n,j}(x)=1$, it suffices to prove that $m_{j+1,n,j}(x)<1$ and $%
m_{j-1,n,j}(x)<1$. By direct calculations we get 
\begin{equation*}
\frac{m_{j,n,j}(x)}{m_{j+1,n,j}(x)}=\frac{j+1}{n+j}\frac{1+x}{x}.
\end{equation*}
Since the function $g(y)=(1+x)/x$ is strictly decreasing on the interval $%
\left[\frac{j}{n-1}, \frac{(j+1)}{n-1}\right]$, it results that $\frac{1+x}{x%
}> \frac{n+j}{j+1}.$ Clearly, this implies $m_{j,n,j}(x)/m_{j+1,n,j}(x)>1$
that is $m_{j,n,j}(x)<1$ By similar reasonings we get that $m_{j-1,n,j}(x)<1$
and the proof is complete.
\end{proof}

Let us consider now a function $f\in C_+([a,b]).$ Combining formula(\ref{f7}%
) with the conclusion of Lemma \ref{12} we can simplify the method to
compute $U_n^{(M)}(f;[a,b])(x)$ for some $x\in [a,b]$. Let us choose $j \in
\left \lbrace 0,1,...,n-2\right \rbrace $ and $x \in
[a+(b-a)j/(n-1),a+(b-a)(j+1)/(n-1)].$ By properties of continuous functions,
an immediate consequence of Lemma \ref{12} is that $m_{k,n,j}(x)\leq 1$ for
all $k\in \left \lbrace 0,1,...,n\right \rbrace .$ This implies that 
\begin{equation}  \label{f8}
\bigvee_{k=0}^n b_{n,k}(x)=b_{n,j}(x), x\in \left[
a+(b-a)j/(n-1),a+(b-a)(j+1)/(n-1)\right].
\end{equation}
Therefore, denoting for each $k\in \left \lbrace 0,1,...,n\right \rbrace $
and

$x\in \left[ a+(b-a)j/(n-1),a+(b-a)(j+1)/(n-1)\right]$ 
\begin{equation}  \label{f9}
f_{k,n,j}(x)=m_{k,n,j}(x).f(a+(b-a)k/n),
\end{equation}
by (\ref{f7}) and (\ref{f8}) we obtain 
\begin{equation}  \label{f10}
U_n^{(M)}(f;[a,b])(x)=\bigvee_{k=0}^nf_{k,n,j}(x), \ x\in \left[
a+(b-a)j/(n-1),a+(b-a)(j+1)/(n-1)\right].
\end{equation}
The above formula generalizes a similar formula from paper \cite{bede2}
where the particular case $a=0, b=1$ is considered. By Lemma \ref{12} we
also note that for any $k\in \left \lbrace 0,1,...,n\right \rbrace $ and $x
\in \left[ a+(b-a)j/(n-1),a+(b-a)(j+1)/(n-1)\right]$, we have $%
f_{k,n,j}(x)\leq f(a+(b-a)k/n).$

\begin{lemma}
\label{lemma13} Let $a,b \in \mathbb{R},$ $a< b.$ If $f:[a,b]\rightarrow 
\mathbb{R}_+$ is bounded then we have $U_{n}^{\left( M\right)}\left(
f;[a,b]\right)(a+j(b-a)/(n-1))\geq f(a+j(b-a)/(n-1))$ for all $j\in \left
\lbrace0,1,\cdots n-2 \right \rbrace.$
\end{lemma}

\begin{proof}
From Lemma \ref{12}, since

$a+j(b-a)/(n-1) \in \left( a+(b-a)j/(n-1),a+(b-a)(j+1)/(n-1)\right)$ and $%
m_{k,n,j}(a+j(b-a)/(n-1))=\frac{b_{n,k}\left(a+j(b-a)/(n-1)\right)}{%
b_{n,j}\left(a+j(b-a)/(n-1)\right)}$ for all $k\in \left \lbrace0,1,\cdots,n
\right \rbrace $ it follows that 
\begin{equation*}
\bigvee_{k=0}^n b_{n,k}\left(a+j(b-a)/n \right)=b_{n,j}\left(a+j(b-a)/n
\right).
\end{equation*}
Then we get 
\begin{equation*}
\begin{split}
U_{n}^{\left( M\right)}\left( f;[a,b]\right)(a+j(b-a)/n)=&\frac{%
\bigvee_{k=0}^n b_{n,k}\left(a+j(b-a)/n \right)f(a+j(b-a)/n)}{
b_{n,j}\left(a+j(b-a)/n \right)} \\
& \geq \frac{b_{n,j}\left(a+j(b-a)/n \right)f(a+j(b-a)/n)}{
b_{n,j}\left(a+j(b-a)/n \right)} \\
& = f(a+j(b-a)/n).
\end{split}%
\end{equation*}
and lemma is proved.
\end{proof}

\begin{theorem}
\label{teo14} Let $u$ be a fuzzy number with supp$\left( u\right) =\left[a,b%
\right]$ and core$\left( u\right) =\left[ c,d\right] $ such that $a\leq
c<d\leq b.$ Then \ for sufficiently large $n,$ it result that $\widetilde{U}%
_{n}^{\left( M\right) }\left( u;\left[a,b\right] \right)$ is a fuzzy number
such that :

\begin{itemize}
\item[i.] $supp(u)=supp(\widetilde{U}_{n}^{\left( M\right) }\left( u;\left[
a,b\right] \right) );$

\item[ii.] if $core(\widetilde{U}_{n}^{\left( M\right) }\left( u;\left[a,b %
\right] \right) )=[c_{n},d_{n}]$, then $c_{n}$ and $d_{n}$ can be determined
precisely and in addition we have $\left \vert c-c_{n}\right \vert \leq
(b-a)/n $ and $\left \vert d-d_{n}\right \vert \leq (b-a)/n$;

\item[iii.] if, in addition, $u$ is continuous on $[a,b]$, then%
\begin{equation*}
\left \vert \widetilde{U}_{n}^{\left( M\right) }\left( u;\left[a,b\right]
\right) -u\left( x\right) \right \vert \leq 6\left( \left[ b-a\right]
+1\right) \omega _{1}\left( u,\frac{1}{\sqrt{n}}\right) _{\left[a,b\right] },
\end{equation*}
\end{itemize}

for all $x\in \mathbb{R}$.
\end{theorem}

\begin{proof}
Let $n\in \mathbb{N}$, such that $(b-a)/n< d-c$.Theorem \ref{t7}, there is $%
c^{\prime }\in [a, b]$ such that $\widetilde{U}_{n}^{\left( M\right) }\left(
u;\left[a, b\right] \right)$ is nondecreasing on $[a,c^{\prime }]$ and
nonincreasing on $[c^{\prime },b].$ Beside, from the description of $%
\widetilde{U}_{n}^{\left( M\right) }\left( u;\left[a,b\right] \right)$, it
gives us that \newline
By indicating $\parallel . \parallel$ the uniform norm on $B([a, b])$ the
space of bounded functions on $[a,b]$, $\parallel \widetilde{U}_{n}^{\left(
M\right) }\left( u;\left[a,b\right] \right)\parallel \leq \parallel
u\parallel$ and since $\parallel u\parallel=1,$ it follows that $\parallel 
\widetilde{U}_{n}^{\left( M\right) }\left( u \right)\parallel \leq 1.$
Consequently, it is sufficient to demonstrate that $\widetilde{U}%
_{n}^{\left( M\right) }\left( u \right)$ is a fuzzy number in order to
obtain existence of $\alpha \in [a,b]$ such that $\widetilde{U}_{n}^{\left(
M\right) }\left( u \right)(\alpha)=1.$ Let $\alpha=a+j(b-a)/n$ where $j$ is
choosen with the property that $c< \alpha <d.$ Such $j$ exists as $%
(b-a)/n<d-c.$ Since $\alpha \in core(u),$ it results $u(\alpha)=1.$
Also,from Lemma \ref{lemma13}, we can write that $\widetilde{U}_{n}^{\left(
M\right) }\left( u;\left[ a,b\right]\right)(\alpha)\geq u(\alpha)$ and
obviously this means that $\widetilde{U}_{n}^{\left( M\right) }\left( u;%
\left[ a,b\right]\right)$ is a fuzzy number. Since we have $F_{n}^{\left(
M\right) }\left( u;\left[ a,b\right]\right)(a)=u(a)$, $U_{n}^{\left(
M\right) }\left( u;\left[ a,b\right]\right)(b)=u(b)$ and the definitions of $%
u$ and $\widetilde{U}_{n}^{\left( M\right) }\left( u;\left[ a,b\right]\right)
$, it follows that $\widetilde{U}_{n}^{\left( M\right) }\left( u;\left[ a,b%
\right]\right)(x)=0$ outside $[a,b].$ Now, by $u(x)> 0$ and $\widetilde{U}%
_{n}^{\left( M\right) }\left( u;\left[ a,b\right]\right)(x)=U_{n}^{\left(
M\right) }\left( u;\left[ a,b\right]\right)(x)$ for all $x\in (a,b),$ we can
obtain that $\widetilde{U}_{n}^{\left( M\right) }\left( u;\left[ a,b\right]%
\right)(x)>0$ for all $x\in (a,b)$ which proves (i).

Now, let us take $n\in \mathbb{N}$ such that $(b-a)/n \leq d-c.$ Then let us
take $k(n,c), k(n,d)\in \left \lbrace 1, \cdots, n-1 \right \rbrace $ be
with the property that $a+(b-a)(k(n,d)-1)/n < c \leq a+(b-a)k(n,c)/n$ and $%
a+(b-a)k(n,c)/n \leq d < a+(b-a)(k(n,d)+1)/n.$ Since $(b-a)/n\leq d-c$ it is
obvious that $k(n,c)\leq k(n,d).$ Also, by the way $k(n,c)$ and $k(n,d)$
were chosen, we observe that $u(a+(b-a)k/n)=1$ for any $k\in \left \lbrace
k(n,c), \cdots, k(n,d)\right \rbrace$ and $u(a+(b-a)k/n)<1$ for any $k\in
\left \lbrace0, \cdots,n \right \rbrace \setminus \left \lbrace k(n,c),
\cdots, k(n,d)\right \rbrace.$ For some $x\in \left[%
a+k(n,c)(b-a)/n,a+(k(n,c)+1)(b-a)/n \right],$ we have 
\begin{equation*}
\widetilde{U}_{n}^{\left( M\right) }\left( u;\left[ a,b\right]%
\right)(x)=\bigvee_{k=0}^nu_{k,n,k(n,c)}(x).
\end{equation*}
We have 
\begin{equation*}
u_{k(n,c),n,k(n,c)}(x)=m_{k(n,c),n,k(n,c)}(%
\zeta)u(a+(b-a)k(n,c)/n)=u(a+(b-a)k(n,c)/n)=1
\end{equation*}
and by the definition of $k(n,c)$ and by Lemma \ref{lemma13} it is obvious
that for any $k\in \left \lbrace0, \cdots,n \right \rbrace,$ we get 
\begin{equation*}
\widetilde{U}_{n}^{\left( M\right) }\left( u;\left[ a,b\right]%
\right)(x)=u(a+(b-\vartheta)k(n,c)/n)=1,
\end{equation*}
$\forall$ $x\in \left[a+k(n,c)(b-a)/n,a+(k(n,c)+1)(b-a)/n \right].$
Similarly we obtain that 
\begin{equation*}
\widetilde{U}_{n}^{\left( M\right) }\left( u;\left[ a,b\right]%
\right)(x)=u(a+(b-a)k(n,d)/n)=1,
\end{equation*}
$\forall$ $x\in \left[a+k(n,d)(b-a)/n,a+(k(n,d)+1)(b-a)/n \right].$ Now let
us choose arbitrarily $x\in \left(a+(k(n,c)-1)(b-a)/n,a+k(n,c)(b-a)/n
\right),$ we have 
\begin{equation*}
\widetilde{U}_{n}^{\left( M\right) }\left( u;\left[ a,b\right]%
\right)(x)=\bigvee_{k=0}^nu_{k,n,k(n,c)-1}(x).
\end{equation*}
If $k\in \left \lbrace k(n,c), \cdots, k(n,d)\right \rbrace,$ then we get 
\begin{equation*}
\begin{split}
u_{k,n,k(n,c)-1}(x)=&m_{k,n,k(n,c)-1}(x)u(a+(b-a)k/n)< u(a+(b-a)k/n) \\
&=u(a+(b-a)k(n,c)/n)=\widetilde{U}_{n}^{\left( M\right) }\left( u;\left[ a,b %
\right]\right)(a+(b-a)k(n,c)/n).
\end{split}%
\end{equation*}

If $k\not \in \left \lbrace k(n,c), \cdots, k(n,d)\right \rbrace,$ then we
get 
\begin{equation*}
\begin{split}
u_{k,n,k(n,c)-1}(x)=&m_{k,n,k(n,c)-1}(x)u(a+(b-a)k/n)\leq u(a+(b-a)k/n) \\
&< u(a+(b-a)k(n,c)/n)=\widetilde{U}_{n}^{\left( M\right) }\left( u;\left[ a,b%
\right]\right)(a+(b-a)k(n,c)/n).
\end{split}%
\end{equation*}
From the propertiy of quasi-concavity of $\widetilde{U}_{n}^{\left( M\right)
}\left( u;\left[ a,b\right]\right)$ on $[a,b]$ it easily results that 
\begin{equation*}
\widetilde{U}_{n}^{\left( M\right) }\left( u;\left[ a, b\right]\right)(x)< 
\widetilde{U}_{n}^{\left( M\right) }\left( u;\left[ a, b\right]%
\right)(a+(b-a)k(n,c)/n), \  \forall x\in [a,a+k(n,c)(b-a)/n)].
\end{equation*}
Similarly, we get 
\begin{equation*}
\widetilde{U}_{n}^{\left( M\right) }\left( u;\left[ a, b\right]\right)(x)< 
\widetilde{U}_{n}^{\left( M\right) }\left( u;\left[ a,b\right]%
\right)(a+(b-a)(k(n,d)+1)/n),
\end{equation*}
$\forall$ $x\in [a+(k(n,d)+1)(b-a)/n),b]$. From the above inequalities, the
fact that 
\begin{equation*}
\begin{split}
\widetilde{U}_{n}^{\left( M\right) }\left( u;\left[ a,b\right]%
\right)(a+(b-a)k(n,c)/n)=&u(a+(b-a)k(n,c)/n)=u(a+(b-a)k(n,d)/n) \\
=&\widetilde{U}_{n}^{\left( M\right) }\left( u;\left[ a,b\right]%
\right)(a+(b-a)(k(n,d)+1)/n)=1,
\end{split}%
\end{equation*}
we get that $\widetilde{U}_{n}^{\left( M\right) }\left( u;\left[ a,b\right]%
\right)$ reaches its maximum value only in the range $[a+(b-a)k(n,c)/n,
a+(b-a)(k(n,d)+1)/n]$ which by the description of $\widetilde{U}_{n}^{\left(
M\right) }\left( u;\left[ a,b\right]\right)$ implies that $core\left(%
\widetilde{U}_{n}^{\left( M\right) }\left( u;\left[ a, b\right]%
\right)\right)=\left[a+(b-a)k(n,c)/n, a+(b-a)(k(n,d)+1)/n\right].$ Then,
indicating $c_n=a+(b-a)k(n,c)/n$ one can see that both $c_n$ and $c$ belong
to the interval $[a+(b-a)(k(n,c)-1)/n,a+(b-a)k(n,c)/n]$ of length $(b-a)/n$
and hence $|c-c_n|\leq (b-a)/n.$ Correlatively, indicating $d_n
=a+(b-a)(k(n,d)+1)/(n+1)$ we obtain that $|d-dn|\leq (b-a)/n$ and the proof
of statement (ii) is complete. (iii) The proof is immediate by Theorem \ref%
{t6}, taking into account the continuity of $u$.
\end{proof}

\end{document}